\documentclass[11pt, oneside]{amsart}

\usepackage{amssymb}
\usepackage{amsmath}
\usepackage{amsthm}
\usepackage{enumerate}
\usepackage{enumitem}
\usepackage{multirow}
\usepackage{graphicx}

\theoremstyle{plain}
\newtheorem{theorem}{Theorem}[section]
\newtheorem{lemma}[theorem]{Lemma}

\newtheorem{definition}[theorem]{Definition}

\newtheorem{remark}[theorem]{Remark}

\theoremstyle{remark}

\title{Tame torsion, the tame inverse Galois problem, and endomorphisms}
\author{Matthew Bisatt}
\address{Fry Building, University of Bristol, Bristol, BS8 1UG, UK}
\email{matthew.bisatt@bristol.ac.uk}

\date{\today}

\begin{document}
\global\long\def\OF{\mathcal{O}_F}
\global\long\def\cK{\mathcal{K}}
\global\long\def\cR{\mathcal{R}}
\global\long\def\cC{\mathcal{C}}
\global\long\def\cJ{\mathcal{J}}
\global\long\def\Q{\mathbb{Q}}
\global\long\def\C{\mathbb{C}}
\global\long\def\qp{\mathbb{Q}_p}
\global\long\def\Z{\mathbb{Z}}
\global\long\def\fp{\mathbb{F}_p}
\global\long\def\Jac{\operatorname{Jac}}
\global\long\def\End{\operatorname{End}}
\global\long\def\Aut{\operatorname{Aut}}
\global\long\def\Gal{\operatorname{Gal}}
\global\long\def\GSp{\operatorname{GSp}}
\global\long\def\Sp{\operatorname{Sp}}
\global\long\def\im{\operatorname{Im}}
\global\long\def\AJ{\operatorname{AJ}}

\begin{abstract}
	Fix a positive integer $g$ and rational prime $p$. We prove the existence of a genus $g$ curve $C/\Q$ such that the mod $p$ representation of its Jacobian is tame by imposing conditions on the endomorphism ring. As an application, we consider the tame inverse Galois problem and are able to realise general symplectic groups as Galois groups of tame extensions of $\Q$.
\end{abstract}

\maketitle

\section{Introduction}

Let $C/\Q$ be a nice (i.e. smooth, projective and geometrically integral) curve of genus $g$ with Jacobian $J_C$. Fix a prime $p$. Then it is well known that the subgroup of points of $J_C$ of order dividing $p$, $J_C[p]$, defines a finite Galois extension $\Q(J_C[p])/\Q$. We would like to control the ramification of this extension as much as possible. 

We say that a number field $F$ is \emph{tame} (resp. \emph{unramified}) if $F/\Q$ is tamely ramified (resp. \emph{unramified}) at every finite prime of $F$. Throughout, we will always write $\zeta_m$ for a primitive $m$-th root of unity.

\begin{remark}
Let $\zeta_p$ be a primitive $p$-th root of unity. Recall that by the Weil pairing, we have $\Q(\zeta_p) \subset \Q(J_C[p])$. If $p$ is odd, then $\Q(\zeta_p)/\Q$ is (tamely) ramified at $p$ and hence $\Q(J_C[p])$ cannot be unramified over $\Q$.\footnote{In fact, the only unramified number field is $\Q$ by the theorem of Hermite--Minkowski.}

If $p=2$, then we can ensure that $\Q(J_C[2])=\Q$ (which is unramified), for example by using the genus $g$ hyperelliptic curve $C: y^2=\prod\limits_{j=0}^{2g}(x-j)$.
\end{remark}

For our first result, we address the question of whether the extension $\Q(J_C[p])$ can be tame.

\begin{theorem}
\label{mainthm}
	Fix a positive integer $g$ and rational prime $p$. Then there exists a nice curve $C/\Q$ of genus $g$ such that $\Q(J_C[p])/\Q$ is tame.
\end{theorem}

As an application of this result, we study the tame inverse Galois problem. Recall that the classical inverse Galois problem asks if every finite group $G$ is (abstractly) realisable as a Galois group $\Gal(F/\Q)$ for some Galois extension $F/\Q$. The tame variant (due to Birch \cite[p.35]{Bir94}) further asks if we can moreover impose that $F$ is tame. We consider the tame version for the case $G=\GSp_{2g}(\fp)$, $p$ odd.

This is known when $g\!=\!1$ (all $p$) and~$g\!=\!2$ ($p \geqslant 5$) thanks to the work of Arias-de-Reyna--Vila \cite[Theorem 1.2]{AV09}, \cite[Theorem 5.3]{AV11}. Recent work of the author with T. Dokchitser \cite[Theorem 6.7]{BD19} strengthens this result to hold for any genus $g$ and odd prime $p$, dependent on a variant of Goldbach's conjecture; in particular it holds for $g \leqslant 10^7$ and any odd $p$ by computational verification.

Our second result is an unconditional version of this but the method of proof is less explicit so would not enable one to construct such a curve directly.

\begin{theorem}[=Theorem \ref{tameIGP}]
\label{tamethm}
	Fix a positive integer $g$ and odd prime $p$. There exists a nice curve $C/\Q$ of genus $g$ such that the extension $\Q(J_C[p])/\Q$ is tame, and $\Gal(\Q(J_C[p])/\Q) \cong \GSp_{2g}(\fp)$.
\end{theorem}

\begin{remark}
	Before we start our proofs, we note that is also an alternative proof in \cite{BD19} to Theorem \ref{mainthm} with an identical approach for $\ell \neq p$ to force $\Q(J[p])$ to be tame. However the approach for $\ell=p$ is substantially different; they use the theory of Mumford curves whereas here we will impose restrictions on the endomorphism algebra. Moreover, they treat the case $p=2$ separately which is not necessary here.
\end{remark}

\noindent \textbf{Layout of the paper.} In \S \ref{localS}, we give Kisin's result on local constancy in $p$-adic families which enables us to reduce our global problem to a local problem, and provide a proof assuming the endomorphism ring of a particular family of curves. In \S \ref{CMS}, we recap the theory of abelian varieties with complex multiplication and explain how this enables us to ensure tame ramification at $p$. We then give an explicit family of curves whose Jacobians have suitable endomorphism rings in \S \ref{EndS} which completes the proof of Theorem \ref{mainthm}. Finally in \S5, we turn our attention to the inverse Galois problem and give a proof of Theorem \ref{tamethm}.

\vspace*{5pt}

\noindent \textbf{Acknowledgements.} I would like to thank Vladimir Dokchitser for suggesting this approach and helpful conversations. I am also indebted to the referee for suggesting the proof of Theorem \ref{tameIGP}.

\section{Local strategy}
\label{localS}

We now discuss conditions to ensure that $\Q(J_C[p])/\Q$ is tame. Our approach will be to impose some local conditions on $C$ in order to guarantee tameness at each prime and then apply the following result of Kisin which allows us to amalgamate these local conditions using the Chinese remainder theorem.

\begin{theorem}
\label{kisin}
Let $\ell$ be a prime. Let $C_f: y^2=f(x)$ be a hyperelliptic curve, with $f\in\Z_{\ell}[x]$ squarefree. For every $m\geqslant 1$, there exists $N\geqslant 1$ such that if $\tilde f\equiv f\mod{\ell^N}$ and $\deg(f)=\deg(\tilde f)$ then $C_{\tilde f}: y^2=\tilde f(x)$ is a hyperelliptic curve with $$\Jac(C_{\tilde f})[m] \cong \Jac(C_f)[m]$$ as $\Gal(\overline{\Q_{\ell}}/\Q_{\ell})$-modules.
\end{theorem}

\begin{proof}
This is a special case of \cite[Theorem 5.1(1)]{Kis99}. Note that for $N$ large enough, all $\tilde f\equiv f\mod{\ell^N}$ are squarefree, and so define a $\ell$-adic family of hyperelliptic curves of the same genus.
\end{proof}

We begin with giving the local conditions away from $p$.

\begin{lemma}
\label{local}
	Fix a rational prime $p$. Let $C/\Q$ be a genus $g$ curve with Jacobian $J_C$ such that:
	\begin{enumerate}
		\item $C$ has good reduction at all $\ell \leqslant 2g+1, \ell \neq p$;
		\item $\Q_p(J_C[p])/\Q_p$ is tamely ramified. 
	\end{enumerate}
	Then $\Q(J_C[p])/\Q$ is tame.
\end{lemma}

\begin{proof}
	We have to show that $\Q(J_C[p])/\Q$ is tamely ramified at $\ell$ for all primes $\ell \neq p$. If $\ell \leqslant 2g+1$ then $C$, and hence $J_C$, has good reduction at $\ell$ by assumption. The criterion of N\'{e}ron--Ogg--Shafarevich then implies that $\Q(J_C[p])/\Q$ is unramified at $\ell$. On the other hand, if $\ell>2g+1$, then the extension is always tamely ramified at $\ell$ by a result of Serre--Tate \cite[p.497]{ST68}.
\end{proof}

It remains to provide a condition at $p$ to ensure tameness. We do this using endomorphism algebras through the following result.

\begin{theorem}[=Theorem \ref{endpf}]
\label{endthm}
	Fix a positive integer $g$ and rational prime $p$. Choose $n \in \{2g+1,2g+2\}$ such that $p \nmid n$. Let $C_n/\Q$ be the hyperelliptic curve $y^2=x^n-1$ with Jacobian $J_n$. Then $\Q_p(J_n[p])/\Q_p$ is tamely ramified.
\end{theorem}

We are now in a position to prove Theorem \ref{mainthm}, assuming the above theorem. The remainder of the paper will then be dedicated to providing a proof of Theorem \ref{endthm}.

\begin{proof}[Proof of Theorem \ref{mainthm}]
	We will construct a family of hyperelliptic curves $C_f:y^2=f(x)$ of genus $g$ which satisfy this condition, where $\deg(f)=2g+2$. 
	
	For primes $\ell \leqslant 2g+1$, $\ell \neq p$ choose a polynomial $f_{\ell}(x) \in \Z[x]$ of degree $2g+2$ such that $C_{f_{\ell}}$ has good reduction at $\ell$. If $\ell$ is odd, then it suffices to choose $f_{\ell}$ such that its reduction $\mod{\ell}$ has the same degree and separable; if $f \equiv f_{\ell} \mod{\ell}$, then $C_f$ will also have good reduction at $\ell$. For $\ell=2$, explicit conditions for $f_2$ are given in \cite[Lemma 7.7]{AD19}; note that $f \equiv f_2 \mod{2}$ is insufficient but $f \equiv f_2 \mod{2^{2g+2}}$ will suffice in condition (ii) of \emph{loc. cit}.
	
	For $\ell=p$, choose $n \in \{2g+1,2g+2\}$ such that $p \nmid n$ and suppose $f \equiv x^n-1 \mod{p^N}$ where $N$ is as in Theorem \ref{kisin} (if necessary, change $C_n:y^2=x^n-1$ to an even degree model). Now let $f(x)$ be a squarefree polynomial of degree $2g+2$ satisfying the above congruence conditions and let $C_f:y^2=f(x)$. Then by Theorems \ref{kisin} and \ref{endthm}, $\qp(\Jac(C_f)[p])/\qp$ is tamely ramified. The result now follows from Lemma \ref{local}.	
\end{proof}

\begin{remark}
\label{sqfree}
	Observe that for $p \neq 2$, the hyperelliptic curve $C_n$ has good reduction at $p$ for $p \nmid n$. Hence it is possible to extend this approach to constructing curves $C/\Q$ such that $\Q(J_C[m])/\Q$ is tame whenever $m$ is odd and squarefree. We do not prove this more general statement here however but instead refer the reader to \cite{BD19}.
\end{remark}

\section{Abelian varieties with complex multiplication}
\label{CMS}

In order to obtain a sufficiently large endomorphism ring, we consider abelian varieties with complex multiplication. In this section, we will provide the relevant background and show how this can give us the tame torsion result we want; in the next section will realise this as the Jacobian of a curve.

\begin{definition}
Let	$A/\qp$ be an abelian variety of dimension $g$. Let $F$ be an \'{e}tale $\Q$-algebra of dimension $2g$. We say that $A$ has complex multiplication by an order $R \subset F$ if $$\End_{\overline{\qp}}(A) \cong R.$$
\end{definition}

Note that if $A/\qp$ is absolutely simple, then $F$ is a number field and $R$ has finite index inside the ring of integers $\OF$. In general $F=\prod F_i$ will be a finite product of number fields, so we will write $\OF=\prod \mathcal{O}_{F_i}$ for the maximal order of $F$ and will restrict ourselves to the setting where $R=\OF$.

The benefit of having lots of endomorphisms is that it heavily restricts the possible choices of $\qp(A[p])$; in fact, the Galois group of the $p$-torsion field is abelian once all the endomorphisms are defined over the base field of $A$.

\begin{theorem}
\label{image}
	Let $\cK/\qp$ be a finite extension and let $A/\cK$ be an abelian variety with complex multiplication by the maximal order $\OF \subset F$. Suppose that all endomorphisms of $A$ are defined over $\cK$, i.e. $\End_{\cK}(A)=\End_{\overline{\qp}}(A)$. Let $p$ be a prime and let $\rho$ be the mod $p$ Galois representation of $A/\cK$. Then $\im(\rho) \subset (\OF \otimes_{\Z} \fp)^{\times}$.
\end{theorem}

\begin{proof}
This is essentially Corollary 2 of \cite[p.502]{ST68} in conjunction with the remark proceeding it. Our assumption that we have complex multiplication by $\OF$ implies that the $p$-adic Tate module $T_pA$ of $A$ is free $\OF \otimes_{\Z} \Z_p$-module of rank 1 and hence the image of the Galois representation attached to $T_pA$ is contained in $(\OF \otimes_{\Z} \Z_p)^{\times}$; tensoring with $\Z/p\Z$ yields the result.
\end{proof}

We now compute $\OF \otimes_{\Z} \fp$ explicitly in order to deduce information about the cardinality of the image of the mod $p$ representation.

\begin{lemma}
\label{decomp}
	Let $F$ be a number field and let $p$ be a prime. Then $$\OF \otimes_{\Z} \fp \cong \prod\limits_{v|p} \dfrac{\mathbb{F}_v[\epsilon]}{\epsilon^{e(v|p)}},$$ where $\mathbb{F}_v$ is the residue field of $\OF$ at $v$, $e(v|p)$ is the ramification degree and the product is over places $v$ of $F$ dividing $p$.
\end{lemma}

\begin{proof}
This follows from the fact that $\OF \otimes_{\Z} \fp \cong \dfrac{\OF}{p\OF}$.
\end{proof}

\begin{lemma}
\label{CMlem}
	Let $A/\qp$ be an abelian variety with complex multiplication by the ring of integers $\OF$ of an \'{e}tale $\Q$-algebra $F=\prod F_i$ with $F_i$ number fields. Let $\cK/\qp$ be a finite extension such that $\End_{\cK}(A)=\End_{\overline{\qp}}(A)$. Suppose that:
	\begin{enumerate}
		\item $p$ is unramified in $F_i/\Q$ for all $i$;
		\item $\cK/\qp$ is tamely ramified.
	\end{enumerate}
	Then $\qp(A[p])/\qp$ is tamely ramified.
\end{lemma}

\begin{proof}
	We prove the stronger statement that $\cK(A[p])/\qp$ is tamely ramified. By our assumption, it suffices to show that $\cK(A[p])/\cK$ is tamely ramified. Since $A$ has complex multiplication by $\OF$, $\Gal(\cK(A[p])/\cK)$ is contained in a subgroup of $(\OF \otimes_{\Z} \fp)^{\times}$ by Theorem \ref{image}. However, $\OF \otimes_{\Z} \fp$ is a product of finite fields of characteristic $p$ by Lemma \ref{decomp} since $p$ is unramified in $F_i$ for all $i$. Hence $(\OF \otimes_{\Z} \fp)^{\times}$ has order coprime to $p$.
	
	Therefore the order of $\Gal(\cK(A[p])/\cK)$ is also coprime to $p$ and hence the extension is tamely ramified at $p$ as the image of wild inertia is a $p$-group.
\end{proof}

\section{Constructing the endomorphism algebra}
\label{EndS}

We now construct a suitable curve whose Jacobian has complex multiplication in order to be able to apply the above results. We do this by explicitly computing the endomorphism ring for a particular class of curves; we closely follow the ideas in \cite[\S 4]{Zar05} but adapt them for a more general cyclic cover.

Throughout this section, we let $C_n/\qp: y^2=x^n-a$ be a hyperelliptic curve, $a \in \qp^{\times}$, and we will write $J_n$ for its Jacobian.

\begin{lemma}
	Let $C_n: y^2=x^n-a$ be a hyperelliptic curve, $a \in \qp^{\times}$, with Jacobian $J_n$. Let $\delta_n$ denote both the $\overline{\qp}$-automorphism of $C_n$ given by $(x,y) \mapsto (\zeta_nx,y)$ and the corresponding automorphism of $J_n$ given by Albanese functoriality. Then the subring $\mathbb{Z}[\delta_n] \subset \End_{\overline{\qp}}J_n$ is isomorphic to $\mathbb{Z}[t]/P_n(t)$, where $P_n(t) \in \mathbb{Z}[t]$ be the polynomial defined by $$P_n(t)=\begin{cases} \dfrac{t^n-1}{t-1}=t^{n-1}+\cdots+t+1 \qquad &\text{if $n$ is odd}; \\[6pt] \dfrac{t^n-1}{t^2-1}=t^{\frac{n}{2}-1}+\cdots+t+1 \qquad &\text{if $n$ is even}.
\end{cases}$$
\end{lemma}

\begin{proof}
	Let $g$ be the genus of $C_n$ and let $\Omega^1(C_n)$ be the vector space of differentials of the first kind of $C_n$. Recall that a basis of $\Omega^1(C_n)$ is given by $\omega_i=x^i\frac{dx}{y}$, $0\leqslant i \leqslant g-1$. Moreover, these are eigenvectors for the corresponding automorphism $\delta_n^*$ on $\Omega^1(C_n)$ (given by functoriality) with eigenvalues $\zeta_n^{i+1}$; note that these are never $1$ nor $-1$ as $2g<n$.
	
	One can now check that $P_n(\delta_n^*)=0$ in $\End(\Omega^1(C_n))$ and moreover that $P_n$ is the minimal polynomial for $\delta_n^*$ in $\Q[t]$.
	
	Fix $P_0=(0,\sqrt{a}) \in C_n(\overline{\qp})$ which is $\delta_n$-invariant and define the Abel--Jacobi map $\AJ:C_n \rightarrow J_n$ using $P_0$. The induced map $\AJ^*:\Omega^1(J_n) \rightarrow \Omega^1(C_n)$ is then an isomorphism which commutes with $\delta_n$. Hence $\delta_n$ and $\delta_n^*$ have the same minimal polynomial $P_n(t)$ in their respective endomorphism rings. 
\end{proof}

\begin{theorem}
\label{end}
	Let $C_n/\qp: y^2=x^n-a$ be a hyperelliptic curve, $a \in \qp^{\times}, n \geqslant 3$, with Jacobian $J_n$. Then $$\End_{\overline{\qp}} (J_n) \cong \prod\limits_{\substack{d \mid n \\ d>2}} \Z[\zeta_d].$$
\end{theorem}

\begin{proof}
	By the above lemma, $\Z[\delta_n] \cong \prod\limits_{\substack{d \mid n \\ d>2}} \Z[\zeta_d]$ so it just remains to prove the equality. Note that $J_n$ is an abelian variety with complex multiplication by $F=\Q[\delta_n] \cong \prod\limits_{\substack{d \mid n \\ d>2}} \Q(\zeta_d)$ (this has dimension $2g$ over $\Q$). Since $\Z[\delta_n]$ is the maximal order of $F$, we have equality.
\end{proof}

We are now finally in a position to prove Theorem \ref{endthm}.

\begin{theorem}[=Theorem \ref{endthm}]
\label{endpf}
	Fix a positive integer $g$ and rational prime $p$. Choose $n \in \{2g+1,2g+2\}$ such that $p \nmid n$. Let $C_n/\Q$ be the hyperelliptic curve $y^2=x^n-1$ with Jacobian $J_n$. Then $\Q_p(J_n[p])/\Q_p$ is tamely ramified.
\end{theorem}

\begin{proof}
	First note that such a choice of $n$ is always possible. Now the endomorphism algebra is $F=\prod\limits_{\substack{d \mid n \\ d>2}} \Q(\zeta_d)$ and note that each subfield of $F$ is contained in $\Q(\zeta_n)$. As $p \nmid n$, $p$ is unramified in $\Q(\zeta_n)$.
	
	Moreover, we can see by construction that all endomorphisms are defined over $\cK=\qp(\zeta_n)$ which is tamely ramified. The theorem now follows from Lemma \ref{CMlem}.
\end{proof}

\section{The tame inverse Galois problem}

We finally focus on the tame inverse Galois problem using the mod $p$ representation of an abelian variety to realise $\GSp_{2g}(\fp)$. Since we have restricted our attention to hyperelliptic curves, the image of the mod $2$ representation is contained $S_{2g+2}$ and hence cannot be surjective for $g \geqslant 3$; for this reason we assume $p$ to be odd.

The approach will be to show that the set of hyperelliptic curves which satisfy our conditions has positive density thus proving existence. For readability, we quickly define the notion of density.

\begin{definition}
Let $B \subset \Z^n$. The density of $B$ is $$\lim\limits_{X \rightarrow \infty} \dfrac{\# \{(a_1,\cdots, a_n) \in B : |a_i| \leqslant X \text{ for all } 1\leqslant i \leqslant n \}}{\# \{(a_1,\cdots, a_n) \in \Z^n : |a_i| \leqslant X \text{ for all } 1\leqslant i \leqslant n \}},$$ if the limit exists.
\end{definition}

\begin{theorem}
\label{tameIGP}
	Fix a positive integer $g$ and odd prime $p$. There exists a nice curve $C/\Q$ of genus $g$ such that the extension $\Q(J_C[p])/\Q$ is tame, and $\Gal(\Q(J_C[p])/\Q) \cong \GSp_{2g}(\fp)$.
\end{theorem}

\begin{proof}
First identify $\mathbb{A}_{\Q}^{2g+3}$ with the set of polynomials of degree at most $2g+2$ via $a=(a_0,...,a_{2g+2}) \mapsto f_a:=\sum\limits_{i=0}^{2g+2}a_ix^i$. Now let $U \subset \mathbb{A}_{\Q}^{2g+3}$ be the open variety such that $f_a$ is separable of degree $2g+2$. We then have a smooth family of hyperelliptic curves $\cC \rightarrow U$, where the fibre above $a \in U$ is the curve $\cC_a: y^2=f_a(x)$. Let $\mathcal{J} \mapsto U$ be the abelian scheme corresponding to the Jacobian, i.e. the fibre above $a$, $\cJ_a$, is the Jacobian of $\cC_a$.

Observe that the $p$-torsion subscheme $\mathcal{J}[p]$ is an \'{e}tale sheaf over $U$ and defines a representation of the \'{e}tale fundamental group $$\rho: \pi^{\acute{e}t}_1(U,u_0) \rightarrow \GSp_{2g}(\fp),$$
where $u_0$ is a fixed basepoint for $U$.\footnote{See for example \cite[\S4.6]{Sza09} for background on the \'{e}tale fundamental group or \cite[\S2]{Yel15} for a description of the analogous representation for the topological fundamental group in the complex setting.} We claim $\rho$ is surjective.

Indeed, by the Weil pairing, it suffices to show that $\rho(\pi^{\acute{e}t}_1(U_{\C},u_0)) = \Sp_{2g}(\fp)$, where $U_{\C}$ is the base change of $U$ to $\C$. Since $p$ is odd, this follows from \cite[Th\'{e}or\`{e}me 1(1)]{AC79}; A'Campo uses the topological fundamental group but we note that the \'{e}tale fundamental group is isomorphic to the profinite completion of the topological one in the complex setting by Riemann's existence theorem.

By Hilbert's irreducibility theorem, the set of polynomials $f_a$ with integer coefficients (corresponding to points $a \in U(\Q) \cap \Z^{2g+3}$) such that the specialisation of $\rho$ at $a$ induces an isomorphism $\Gal(\Q(\cJ_a[p])/\Q) \cong \GSp_{2g}(\fp)$ has density $1$ inside $\Z^{2g+3}$ (see for example \cite[p.123]{Ser97}). From the proof of Theorem \ref{mainthm}, we see that there is a congruence class of polynomials $f_a$ such that $\Q(\cJ_a[p])/\Q$ is a tame extension and hence the intersection of these two sets has positive density in $\Z^{2g+3}$ which proves the result.
\end{proof}

\begin{remark}
It is possible to extend the above proof to realise $\GSp_{2g}(\Z/m\Z)$ by looking at $m$-torsion instead when $m$ is odd and squarefree. Indeed, tameness is done similarly (cf. Remark \ref{sqfree}) and a refinement of Hilbert's irreducibility theorem enables us to ensure the $p$-torsion fields are disjoint for all $p \mid m$.
\end{remark}

\bibliographystyle{alpha}
\bibliography{endos}
\end{document}